\documentclass[11pt]{article}

\usepackage{amssymb}
\usepackage{amsmath}
\usepackage{amsthm}
\usepackage{enumerate}
\usepackage{graphicx}
\usepackage{caption}
\usepackage{subfigure}
\usepackage[usenames]{color}
\usepackage{xcolor}
\usepackage{hyperref}

\hypersetup{
            breaklinks=true,%
            pdfstartpage=1,%
            letterpaper=true,%
			bookmarks=true,%
            bookmarksnumbered=false,%
			plainpages=false,%
			naturalnames=true,%
linkbordercolor=blue,
               urlcolor=blue,%
			citecolor=blue,%
			pdftitle={Project Description},%
			pdfauthor={Andrzej Dudek}}

\begingroup\makeatletter\ifx\SetFigFont\undefined%
\gdef\SetFigFont#1#2#3#4#5{%
  \reset@font\fontsize{#1}{#2pt}%
  \fontfamily{#3}\fontseries{#4}\fontshape{#5}%
  \selectfont}%
\fi\endgroup%

\newtheoremstyle{plain}%
    {8pt plus2pt minus4pt}%
    {8pt plus2pt minus4pt}%
    {\itshape}%
    {}%
    {\bfseries\scshape}%
    {}%
    {6pt}
    {}%

\newtheoremstyle{remark}%
    {8pt plus2pt minus4pt}%
    {8pt plus2pt minus4pt}%
    {\upshape}
    {}%
    {\bfseries\scshape}%
    {}%
    {6pt}
    {}%

\theoremstyle{plain}
\newtheorem{thm}{Theorem}[section]
\newtheorem*{thm2}{Theorem}

\newtheorem{prop}[thm]{Proposition}

\newtheorem{ques}[thm]{Question}

\theoremstyle{remark}

\newcommand{\eR}{\hat{R}} 
\newcommand{\EX}[1]{\mathbb{E}(#1)}
\newcommand{\PR}[1]{\Pr\left(#1\right)} 
\newcommand{\bin}[1]{\text{Bin}\left(#1\right)}
\newcommand{\aas}{\emph{a.a.s.}}
\newcommand{\mG}{\mathcal{G}}
\newcommand{\mH}{\mathcal{H}}
\title{On the size-Ramsey number of hypergraphs}

\author{
{\Large{Andrzej Dudek}}\thanks{
\footnotesize {Department of Mathematics, Western~Michigan~University, Kalamazoo, MI 49008, \texttt{andrzej.dudek@wmich.edu}.
Supported in part by Simons Foundation Grant \#244712.}}
\and
{\Large{Steven La Fleur}}\thanks{
\footnotesize {Department of Mathematics and Computer Science, Emory University, Atlanta, GA 30322, \texttt{ slafleu@emory.edu.}}}
\and
{\Large{Dhruv Mubayi}}\thanks{
\footnotesize {Department of Mathematics, Statistics, and Computer Science, University of Illinois at Chicago, Chicago, IL 60607, \texttt{mubayi@math.uic.edu.} Supported in part by NSF grant DMS-1300138}}
\and
{\Large{Vojtech R\"odl}}\thanks{
\footnotesize {Department of Mathematics and Computer Science, Emory University, Atlanta, GA 30322, \texttt{rodl@mathcs.emory.edu.} Supported in part by NSF grants DMS-1301698 and DMS-1102086.}}
}

\date{\today}

\begin{document}
\maketitle		
		
\begin{abstract}
The size-Ramsey number of a graph $G$ is the minimum number of edges in a graph $H$ such that every 2-edge-coloring of $H$ yields a monochromatic copy of $G$. Size-Ramsey numbers of graphs have been studied for almost 40 years with particular focus on the case of trees and bounded degree graphs. 

We initiate the study of size-Ramsey numbers for $k$-uniform hypergraphs.  Analogous to the graph case, we consider the size-Ramsey number of  cliques, paths, trees, and bounded degree hypergraphs. Our results suggest that size-Ramsey numbers for hypergraphs are extremely difficult to determine, and many open problems remain.
\end{abstract}

\section{Introduction}

Given graphs $G$ and $H$, say that $H \to G$
if every 2-edge-coloring of $H$ results in a monochromatic copy of $G$ in $H$. Using this notation, the Ramsey number $R(G)$ of $G$ is the minimum $n$ such that $K_n \to G$. Instead of minimizing the number of vertices, one can minimize the number of edges. Define the {\it size-Ramsey number} $\eR(G)$ of $G$ to be the minimum number of edges in a graph $H$ such that $H \to G$.
More formally, 
$$
\eR(G) = \min \{ |E(H)| :  H\to G \}.
$$
The study of size-Ramsey numbers was proposed by Erd\H{o}s, Faudree, Rousseau and Schelp~\cite{ErdFauRouSch78} in 1978.  By definition of $R(G)$, we have $K_{R(G)} \to G$.  Since the complete graph on $R(G)$ vertices has ${R(G)\choose 2}$ edges, we obtain the trivial bound
\begin{equation} \label{clique}
\eR(G) \le \binom{R(G)}{2}.
\end{equation}
Chv\'atal (see, \emph{e.g.},~\cite{ErdFauRouSch78}) showed that equality holds in (\ref{clique}) for complete graphs. In other words, 
\begin{equation}\label{eq:chvatal}
\eR(K_n) = \binom{R(K_n)}{2}.
\end{equation}One of the first problems in this area was to determine the size-Ramsey number of the $n$ vertex path $P_n$.  Answering a question of Erd\H{o}s~\cite{Erd81}, Beck~\cite{Beck83} showed that 
\begin{equation}\label{eq:beck}
\eR(P_n) = O(n).
\end{equation}
Since $\eR(G) \ge |E(G)|$ for any graph, Beck's result is sharp in order of magnitude. 
The linearity of the size-Ramsey number of paths was generalized to bounded degree trees by Friedman and Pippenger~\cite{FriPip87} and to cycles by Haxell, Kohayakawa and {\L}uczak~\cite{HaxKohLuc95}.
Beck~\cite{Beck90} asked whether $\eR(G)$ is always linear in the size of $G$ for graphs $G$ of bounded degree. This was settled in the negative by R\"{o}dl and Szemer\'edi~\cite{RodSze2000}, who proved that there are graphs of order $n$, maximum degree 3, and size-Ramsey number $\Omega(n(\log n)^{1/60})$. They also conjectured that for a fixed integer~$\Delta$ there is an $\varepsilon>0$ such that
\[
\Omega(n^{1+\varepsilon}) = \max_{G} \eR(G) = O(n^{2-\varepsilon}),
\]
where the maximum is taken over all graphs $G$ of order $n$ with maximum degree at most~$\Delta$. The upper bound was recently proved by Kohayakawa, R{\"o}dl, Schacht, and Szemer{\'e}di~\cite{KohRodSchSze11}. For further results about the size-Ramsey number see, \emph{e.g}, the survey paper of Faudree and Schelp~\cite{FauSch99}.

Somewhat surprisingly the size-Ramsey numbers have not been studied for hypergraphs, even though classical Ramsey numbers for hypergraphs have been studied extensively since the 1950's (see, e.g., \cite{ErdRad, ErdHajRad}), and more recently~\cite{ConFoxSud2010}. In this paper we initiate this study for $k$-uniform hypergraphs. A \emph{$k$-uniform hypergraph} $\mathcal{G}$ (\emph{$k$-graph} for short) on a vertex set $V(\mathcal{G})$ is a family of $k$-element subsets (called edges) of $V(\mathcal{G})$.  We write $E(\mathcal{G})$ for its edge set.
Given  $k$-graphs $\mathcal{G}$ and $\mathcal{H}$, say that  $\mH \to \mG$
if every 2-edge-coloring of $\mH$ results in a monochromatic copy of $\mG$ in $\mH$.
Define the \emph{size-Ramsey number} $\eR(\mG)$ of a $k$-graph $\mathcal{G}$ as
$$
\eR(\mathcal{G}) = \min \{ |E(\mathcal{H})| : \mathcal{H}\to \mathcal{G}  \}.
$$

\section{Results and open problems}
Motivated by extending the basic theory from graphs to hypergraphs, we prove results for cliques, trees, paths, and bounded degree hypergraphs.

\subsection{Cliques}
 For every $k$-graph $\mG$, we trivially have 
\[
\eR(\mathcal{G}) \le \binom{R(\mathcal{G})}{k},
\]
where $R(\mathcal{G})$ is the ordinary Ramsey number of $\mathcal{G}$.  Our first objective was to generalize (\ref{eq:chvatal}) to 3-graphs, which shows that equality holds for graphs. It is fairly easy to obtain a lower bound for $\eR(\mathcal{K}^{(3)}_n)$ that is quadratic in  
$R(\mathcal{K}^{(3)}_n)$, but we were only able to do slightly better.

\begin{thm} \label{cliquethm}
${\eR(\mathcal{K}_n^{(3)}) \ge \frac{n^2}{96}\binom{R(\mathcal{K}_n^{(3)})}{2}}$.
\end{thm}

The following basic questions remain open.

\begin{ques} \label{ques:1}
Is $\eR(\mathcal{K}_n^{(k)}) = \binom{R(\mathcal{K}_n^{(k)})}{k}$?
\end{ques}

\begin{ques} \label{ques:2}
For $k\ge 3$ let $N = R(\mathcal{K}_n^{(k)})$. Define $\mathcal{K}_N^{(k)^-}$ to be the hypergraph obtained from $\mathcal{K}_N^{(k)}$ by removing one edge. Is it true that $\mathcal{K}_N^{(k)^-}\rightarrow \mathcal{K}_n^{(k)}$?
\end{ques}

\noindent
Clearly, the affirmative answer to the latter gives a negative answer to Question~\ref{ques:1}.

\subsection{Trees}
Given integers $1\le \ell < k$ and $n$, a $k$-graph $\mathcal{T}_{n,\ell}^{(k)}$ of order~$n$ with edge set $\{e_1,\dots,e_m\}$ is an \emph{$\ell$-tree}, if for each $2\le j \le m$ we have $|e_j \cap \bigcup_{1\le i < j} e_i| \le \ell $ and
$e_j \cap \bigcup_{1\le i < j} e_i \subseteq e_{i_0}$ for some $1\le i_0 < j$.
We are able to give the following general upper bound for trees.

\begin{thm}\label{thm:ub_tree}
Let $1\le \ell < k$ be fixed integers. Then
\[
\eR(\mathcal{T}_{n,\ell}^{(k)}) = O(n^{\ell+1}).
\]
\end{thm}
One can easily show that this bound is tight in order of magnitude when $\ell=1$ (see Section~\ref{sec:trees}  for details). The situation for $\ell\ge 2$ is much less clear.

\begin{ques} 
Let $2\le \ell < k$ be fixed integers. Is it true that for every $n$ there exists a $k$-uniform $\ell$-tree $\mathcal{T}$ of order at most~$n$ such that 
\[
\eR(\mathcal{T}) = \Omega(n^{\ell+1}).
\]
\end{ques}

Here is another related question pointed out by Fox~\cite{Fox}. Let us weaken the restriction on the edge intersection in the definition of  $\mathcal{T}_{n,\ell}^{(k)}$. Let $\bar{\mathcal{T}}_{n,\ell}^{(k)}$ be a $k$-graph of order~$n$ with edge set $\{e_1,\dots,e_m\}$ such that for each $2\le j \le m$ we have $|e_j \cap \bigcup_{1\le i < j} e_i| \le \ell $.

\begin{ques} Let $2\le \ell < k$ be fixed integers.
Is  $\eR(\bar{\mathcal{T}}_{n,\ell}^{(k)} )$ polynomial in $n$?
\end{ques} 
\noindent

\subsection{Paths}
Given integers $1\leq \ell< k$ and $n \equiv \ell\ (\textrm{mod}\ k-\ell)$, we define
an {\em $\ell$-path $\mathcal{P}_{n,\ell}^{(k)}$} to be  the $k$-uniform hypergraph with vertex set $[n]$ and edge set 
$\{e_1, \ldots, e_m\}$, where $e_i=\{(i-1)(k-\ell)+1,(i-1)(k-\ell)+2,\dots,(i-1)(k-\ell)+k\}$ and $m=\frac{n-\ell}{k-\ell}$. In other words, the edges are intervals of length~$k$ in $[n]$ and consecutive edges intersect in precisely $\ell$ vertices.
The two extreme cases of $\ell=1$ and $\ell=k-1$ are referred to as, respectively, \emph{loose} and
\emph{tight} paths. Clearly every $\ell$-path is also an $\ell$-tree. Thus, by Theorem~\ref{thm:ub_tree} we obtain the following result.
\begin{equation}\label{eq:l_path}
\eR(\mathcal{P}_{n,\ell}^{(k)}) = O(n^{\ell+1}).
\end{equation}
Our first result shows that 
determining the size-Ramsey number of a path $\mathcal{P}_{n,\ell}^{(k)}$ for $\ell \le \frac{k}{2}$ can easily be reduced to the graph case. 
\begin{prop} \label{pathprop} Let $1\le \ell \le \frac{k}{2}$. Then,
\[
\eR(\mathcal{P}_{n,\ell}^{(k)}) \le \eR(P_n) = O(n).
\]
\end{prop}
\noindent
Clearly, this result is optimal.

Determining the size-Ramsey number of a path $\mathcal{P}_{n,\ell}^{(k)}$ for $\ell > \frac{k}{2}$ seems  to be a much harder problem. Here we will only consider tight paths ($\ell = k-1$). By~\eqref{eq:l_path} we get
\begin{equation}\label{pathbound}
\eR(\mathcal{P}_{n,k-1}^{(k)}) = O(n^k).
\end{equation}
The most complicated result of this paper is the following
improvement of (\ref{pathbound}). 
\begin{thm} \label{thm:tightpath}
  \label{thm:tight_path_k}
  Fix $k \geq 3$ and let $\alpha = (k-2)/(\binom{k-1}{2} + 1)$.  Then
	\[
	\eR(\mathcal{P}_{n,k-1}^{(k)}) = O(n^{k-1 - \alpha}(\log{n})^{1 + \alpha}).
	\]
\end{thm}
The gap in the exponent of $n$ between the upper and lower bounds for this problem remains quite large (between 1 and $k-1-\alpha$). We believe that the lower bound is much closer to the truth. Indeed,  
the following question still remains open.
\begin{ques}
Is $\eR(\mathcal{P}_{n,k-1}^{(k)}) = O(n)$?
\end{ques}
\noindent
If true, then since $\eR(\mathcal{P}_{n,\ell}^{(k)}) \le \eR(\mathcal{P}_{n,k-1}^{(k)})$, this would imply the linearity of the size-Ramsey number of all $\ell$-paths.

\subsection{Bounded degree hypergraphs}

Our main result about bounded degree hypergraphs is that their size-Ramsey numbers can be superlinear.
This is proved by extending the methods of R\"odl and Sz\'emer\'edi~\cite{RodSze2000} to the hypergraph case. 

\begin{thm} \label{thm:bounddegree}
Let $k\ge 3$ be an integer. Then there is a positive constant $c=c(k)$ such that for every $n$ there is a $k$-graph $\mathcal{G}$ of order at most~$n$ with maximum degree~$k+1$ such that 
\[
\hat{R}(\mathcal{G}) = \Omega(n (\log n)^c).
\]
\end{thm}

There are several other problems to consider such as finding the asymptotic of the size-Ramsey number of cycles and many other classes of hypergraphs. In general, they seem to be very difficult. Therefore, this paper is the first step towards a better understanding of this concept.

In the next sections  we prove these result
for cliques (Section~\ref{sec:cliques}), trees (Section~\ref{sec:trees}), paths (Section~\ref{sec:paths}), and hypergraphs with bounded degree (Section~\ref{sec:bounded}).

\section{Cliques}\label{sec:cliques}

\noindent
{\bf Proof of Theorem \ref{cliquethm}.}
We show that if $\mathcal{H}$ is a 3-graph with $|E(\mathcal{H})| < \frac{n^2}{96}\binom{R(\mathcal{K}_n^{(3)})}{2}$ for $n\ge 4$, then 
$\mathcal{H}\nrightarrow \mathcal{K}_n^{(3)}$.

Induction on $N = |V(\mathcal{H})|$. If $N < R(\mathcal{K}_n^{(3)})$, then there is a 2-coloring of $K_N^{(3)}$ with no monochromatic $K_n^{(3)}$. Since $\mathcal{H} \subseteq K_N^{(3)}$, this coloring yields a 2-coloring of $\mathcal{H}$ with no monochromatic $K_n^{(3)}$.

Suppose that $N \ge R(\mathcal{K}_n^{(3)})$. Since $|E(\mathcal{H})| < \frac{n^2}{96}\binom{R(\mathcal{K}_n^{(3)})}{2}$, there are $u$ and $v$ in $V(\mathcal{H})$ with $\deg(u,v)=|\{ e\in E(\mathcal{H}) : \{u,v\}\subseteq e \}| < \frac{n^2}{32}$. Otherwise, 
\[
|E(\mathcal{H})| = \frac{1}{3} \sum_{\{u,v\} \in \binom{V(\mathcal{H})}{2}} \deg(u,v)
\ge \frac{1}{3} \binom{N}{2} \frac{n^2}{32}
> |E(\mathcal{H})|,
\]
a contradiction.

Let $u$ and $v$ be such that $\deg(u,v) < \frac{n^2}{32}$. Define $\mathcal{H}_u$ as follows:
\[
V(\mathcal{H}_u) = V(\mathcal{H}) \setminus \{v\}
\]
and
\[
E(\mathcal{H}_u) = \{ e  : v\notin e \in E(\mathcal{H})\}
\cup \left\{ \{u,x,y\} : \{v,x,y\}\in E(\mathcal{H}) \text{ and } \{u,x,y\}\notin E(\mathcal{H})\right\}.
\]
Clearly, $|V(\mathcal{H}_u)| = N-1$ and $|E(\mathcal{H}_u)| \le |E(\mathcal{H})| < \frac{n^2}{96}\binom{R(\mathcal{K}_n^{(3)})}{2}$. By the inductive hypothesis there is a 2-coloring $\chi_u$ of the edges of $\mathcal{H}_u$ with no monochromatic $\mathcal{K}_n^{(3)}$. Let $T=T_1 = N_{\mathcal{H}}(u,v)=\{ w\in V(\mathcal{H}) : \{u,v,w\} \in E(\mathcal{H}) \}$. Thus, $T_1 \subseteq V(\mathcal{H}_u)$ and $|T_1| < \frac{n^2}{32}$. If there exists $S_1 \subseteq T_1$ such that $|S_1| \ge \frac{n}{4}$ and $\mathcal{H}_u[S_1 \cup \{u\}]$ is monochromatic, then set $T_2 = T_1 \setminus S_1$. If there exists $S_2 \subseteq T_2$ such that $|S_2| \ge \frac{n}{4}$ and $\mathcal{H}_u[S_2 \cup \{u\}]$ is monochromatic, then set $T_3 = T_2 \setminus S_2$. We continue this process obtaining 
\[
T = S_1 \cup S_2 \cup \dots \cup S_m \cup U,
\]
where $\mathcal{H}_u[S_i \cup \{u\}]$ is monochromatic, $|S_i| \ge \frac{n}{4}$, and $\mathcal{H}_u[U \cup \{u\}]$ contains only monochromatic cliques of order at most $\frac{n}{4}$.

Now we define a 2-coloring $\chi$ of $\mathcal{H}$. 
\begin{enumerate}[(i)]
\item If $v \notin e$, then $\chi(e) = \chi_u(e)$. 
\item If $v \in e = \{v,x,y\}$ and $u\notin e$, then $\chi(e) = \chi_u(\{u,x,y\})$.
\item If $\{u,v\} \subseteq e = \{u,v,x\}$ and $x \in S_i$, then $e$ takes the opposite color to the color of $\mathcal{H}_u[S_i \cup \{u\}]$.
\item If $\{u,v\} \subseteq e = \{u,v,x\}$ and $x \in U$, then color $e$ arbitrarily. 
\end{enumerate}

Now suppose that there is a monochromatic clique $\mathcal{K} = \mathcal{K}_n^{(3)}$ in $\mathcal{H}$. Such a clique must contain~$v$.
Now there are two cases to consider. If $u \notin V(\mathcal{K})$, then the subgraph of $\mathcal{H}_u$ induced by $V(\mathcal{K})\cup \{u\} \setminus \{v\}$ is also a monochromatic copy of $\mathcal{K}_n^{(3)}$, a contradiction. Otherwise,
$u \in V(\mathcal{K})$. Thus, $V(\mathcal{K}) \setminus \{u,v\} \subseteq T$ and $|V(\mathcal{K}) \setminus \{u,v\}| = n-2$. Observe that $|V(\mathcal{K}) \cap S_i|\le 2$ and $|V(\mathcal{K}) \cap U| < \frac{n}{4}$. But this yields a contradiction
\[
n-2 = |V(\mathcal{K}) \setminus \{u,v\}| < 2m + \frac{n}{4} < 2\frac{\frac{n^2}{32} }{ \frac{n}{4} } + \frac{n}{4} = \frac{n}{2} \le n-2,
\]
for $n\ge 4$.
\qed

\section{Trees}\label{sec:trees}
First for convenience we recall the definition of a hypertree.
Given integers $1\le \ell < k$ and $n$, recall that a $k$-graph $\mathcal{T}_{n,\ell}^{(k)}$ of order~$n$ with edge set $\{e_1,\dots,e_m\}$ is an \emph{$\ell$-tree}, if for each $2\le j \le m$ we have $|e_j \cap \bigcup_{1\le i < j} e_i| \le \ell $ and
$e_j \cap \bigcup_{1\le i < j} e_i \subseteq e_{i_0}$ for some $1\le i_0 < j$.
\bigskip

\noindent
{\bf Proof of Theorem \ref{thm:ub_tree}.}
Fix $1 \le \ell \le k$.  We are to show that
$\eR(\mathcal{T}_{n,\ell}^{(k)}) = O(n^{\ell+1})$.
Recall that a \emph{partial Steiner system $S(t,k,N)$} is a $k$-graph of order $N$ such that each $t$-tuple is contained in at most one edge. 
Due to a result of R\"odl~\cite{Rod} it is known that there is a constant $N_0 = N_0(t,k)$ such that for every $N\ge N_0$ there is an $\mathcal{S} = S(t,k,N)$ with the number of edges satisfying
\begin{equation}\label{eq:steiner}
\frac{9}{10} \cdot \frac{\binom{N}{t}}{\binom{k}{t}} \le |E(\mathcal{S})| \le \frac{\binom{N}{t}}{\binom{k}{t}}
\end{equation}
(see also \cite{Kev,W1,W2,W3} for similar results). It is easy to observe that for $1\le s\le t$ every $s$-tuple is contained in at most $\frac{\binom{N-s}{t-s}}{\binom{k-s}{t-s}}$ edges.

Fix $1\le \ell < k$. Let $N = \lceil c n \rceil + \ell$, where the constant $c$ is defined as
\[
c = \max\left\{ N_0(\ell+1, k), \frac{20}{9}(\ell+1)\binom{k}{\ell+1} \right\}.
\] 
Let $\mathcal{H}$ be a $S(\ell+1,k,N)$ satisfying \eqref{eq:steiner}. Observe that if $\ell +1 = k$, then $\mathcal{H}$ can be viewed as a complete $k$-graph of order~$N$. Clearly, $|E(\mathcal{H})| = O(n^{\ell+1})$. It remains to show that for any $\mathcal{T} = \mathcal{T}_{n,\ell}^{(k)}$ tree, $\mathcal{H}\to \mathcal{T}$.

Define a \emph{degree} of a set $U\subseteq V(\mathcal{H})$ ($1\le |U| < k$) by
\[
\deg(U) = |\{e \in E(\mathcal{H}) : e \supseteq U\}|
\]
and for $E(\mathcal{H})\neq \emptyset$ a \emph{minimum (non-zero) $\ell$-degree} by
\[
\delta_{\ell}(\mathcal{H}) = \min \{ \deg(U): |U| =\ell\ \text{and}\ U\subseteq e\ \text{for some}\ e\in E(\mathcal{H})\}.
\]
First observe that for any 2-coloring of the edges of $\mathcal{H}$, there is a monochromatic sub-hypergraph $\mathcal{F}$ with 
$\delta_{\ell}(\mathcal{F}) \ge n$. Indeed, suppose that $\mathcal{\mathcal{H}}$ is colored with blue and red colors. 
Assume by symmetry that the red hypergraph $\mathcal{R}$ has at least $\frac{1}{2}|E(\mathcal{H})|$ edges. Set $\mathcal{R}_0  =\mathcal{R}$. If there exists $U_0 \subseteq V(\mathcal{R}_0)$ with $\deg_{\mathcal{R}_0}(U_0) < n$, then let $\mathcal{R}_1 = \mathcal{R}_0-U_0$ (we remove $U_0$ and all incident to $U_0$ edges). Now we repeat the process. If there exists $U_1 \subseteq V(\mathcal{R}_1)$ with $\deg_{\mathcal{R}_1}(U_1) < n$, then let $\mathcal{R}_2 = \mathcal{R}_1-U_1$. Continue this way to obtain hypergraphs
\[
\mathcal{R} = \mathcal{R}_0 \supseteq \mathcal{R}_1 \supseteq \mathcal{R}_2 \supseteq \dots \supseteq \mathcal{R}_m,
\]
where either $\delta_{\ell}(\mathcal{R}_m)\ge n$ or $\mathcal{R}_m$ is empty hypergraph. But the latter cannot happen, since the number of removed edges from $\mathcal{R}$ is less than
\[
\binom{N}{\ell} n 
= \binom{N}{\ell+1} \frac{\ell+1}{N-\ell} n 
\le \binom{N}{\ell+1} \frac{\ell+1}{c} 
\le \frac{9}{20} \cdot \frac{\binom{N}{\ell+1}}{\binom{k}{\ell+1}} < \frac{1}{2}|E(\mathcal{H})|.
\]

Now we greedily embed $\mathcal{T}$ into $\mathcal{F} = \mathcal{R}_m$. At every step we have a connected sub-tree $\mathcal{T}_i\subseteq \mathcal{T}$. Assume that we already embedded $i$ edges of $\mathcal{T}$ obtaining $\mathcal{T}_i$. Let $|U| \le \ell$ be such that $U\subseteq e$ for some $e\in E(\mathcal{T}_i)$. Observe that there is always an edge $f \in E(\mathcal{F}) \setminus E(\mathcal{T}_i)$ such that  $f \cap V(\mathcal{T}_i) = U$. Indeed, if $|U|=\ell$, then this is true since
$\deg_{\mathcal{F}}(U)\ge n$ and $|V(\mathcal{T}_i)|<n$ and every $(\ell+1)$-tuple of vertices of $\mathcal{F}$ is contained in at most one edge in $\mathcal{F}$. Otherwise, if $|U|<\ell$, first we find a set $W\subseteq V(\mathcal{F}) \setminus V(\mathcal{T}_i)$ such that $|W| = \ell - |U|$ and $U\cup W$ is contained in an edge of $\mathcal{F}$, and next apply the previous argument to $U\cup W$. Thus, we can extend $\mathcal{T}_i$ to $\mathcal{T}_{i+1}$, as required.
\qed
\bigskip

As mentioned in the introduction, it would be interesting to decide whether Theorem~\ref{thm:ub_tree} is tight up to the hidden constant. This is definitely the case for $\ell=1$. Indeed, let $\mathcal{T}$ be a $k$-uniform star-like tree of order $n$ defined as follows. Assume that $2k-2$ divides $n-1$. $\mathcal{T}$ consists of $\frac{n-1}{2k-2}$ arms $\mathcal{P}_i$ (each with two edges):
$E(\mathcal{P}_i) = \{ \{v,w_1^i,w_2^i,\dots,w_{k-1}^i\}, \{w_{k-1}^i, w_k^i, \dots, w_{2k-2}^i\}   \}$, where
$1\le i\le \frac{n-1}{2k-2}$ and all $w_j^i$ vertices are pairwise different (see Figure~\ref{fig:star}).

\begin{figure}
\begin{center}
\includegraphics[scale=0.6]{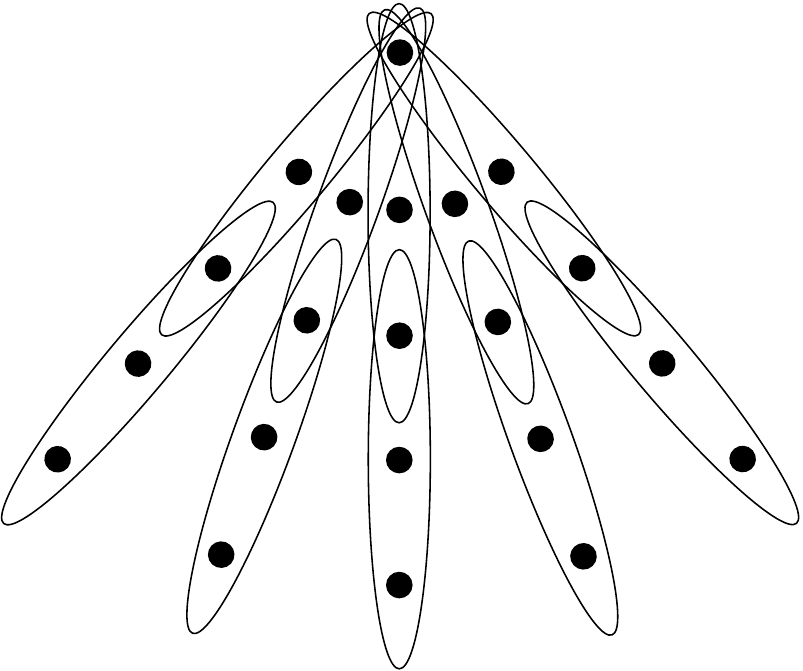}
\caption{A star of order $n$ with $\frac{n-1}{4}$ arms each of length 2.}
\label{fig:star}
\end{center}
\end{figure}

Assume that $\mathcal{H}\to \mathcal{T}$ and color $e\in \mathcal{H}$ by red if degree (in $\mathcal{H}$) of every vertex in $e$ is less than $\frac{n-1}{2k-2}$; otherwise $e$ is blue. Since $\mathcal{H}\to (\mathcal{T})^{e}_2$ and there is no red copy of~$\mathcal{T}$, there must be a blue copy of~$\mathcal{T}$. Every edge in such a copy has at least one vertex of degree at least~$\frac{n-1}{2k-2}$ (in~$\mathcal{H}$). Since $\mathcal{T}$ has $\frac{n-1}{2k-2}$ vertex disjoint edges and every edge  (in~$\mathcal{H}$) can intersect at most 3 of those disjoint edges,
\[
\eR(\mathcal{T}) \ge \frac{1}{3}\cdot \frac{n-1}{2k-2}\cdot \frac{n-1}{2k-2} = \Omega(n^2).
\]

\section{Paths}\label{sec:paths}
In this section we prove Proposition~\ref{pathprop}
and Theorem~\ref{thm:tightpath}.
\bigskip

\noindent
{\bf Proof of Proposition~\ref{pathprop}.}
Let  $H$ be a graph satisfying $H\to P_n$ and $|E(H)| = O(n)$ (\textit{cf.}~\eqref{eq:beck}). 
We construct a $k$-graph $\mathcal{H}$ as follows. Replace every vertex $v\in V(H)$ by an $\ell$-tuple $\{v_1,v_2,\dots,v_{\ell}\}$ (different for every $v$) and each $e=\{v,w\} \in E(H)$ by 
\[
\{v_1,\dots,v_{\ell},w_1,\dots,w_{\ell}, x_1,\dots,x_{k-2\ell}\},
\]
where $x_1,\dots,x_{k-2\ell}$ are different for every edge~$e$, too.
Thus, $\mathcal{H}$ is a $k$-graph with $|V(\mathcal{H})| = \ell |V(H)| + (k-2\ell)|E(H)|$ and $|E(\mathcal{H})| = |E(H)|$. Now color $E(\mathcal{H})$. This coloring (uniquely) defines a coloring of $E(H)$. Since $H$ contains a monochromatic copy of $P_n$, $\mathcal{H}$ also contains a monochromatic copy of $\mathcal{P}_{n,\ell}^{(k)}$. Consequently, $\mathcal{H} \to \mathcal{P}_{n,\ell}^{(k)}$ and the proof is complete.  \qed
\bigskip

We now turn to the main result of this section which we restate for convenience.

\begin{thm2} {\bf \ref{thm:tightpath}}
  \label{thm:tight_path_k}
  Fix $k \geq 3$ and let $\alpha = (k-2)/(\binom{k-1}{2} + 1)$.  Then
	\[
	\eR(\mathcal{P}_{n,k-1}^{(k)}) = O(n^{k-1 - \alpha}(\log{n})^{1 + \alpha}).
	\]
\end{thm2}

First we prove an auxiliary result. In order to do it we state some necessary notation. Set
\[
\beta = \frac{1}{\binom{k-1}{2} + 1}.
\]
For a graph $G=(V,E)$ let $\mathcal{T}_{\ell}(G)$ be the set of all cliques of order $\ell$ and let $t_{\ell} = |\mathcal{T}_{\ell}(G)|$.
Let $A\subseteq V$ and $\mathcal{B}\subseteq \mathcal{T}_{k-1}(G)$ be a family of pairwise vertex-disjoint cliques. Define $x_{A,\mathcal{B}}$ as
the number of $k$-cliques of $G$ which $k-1$ vertices form a vertex set of some $B\in \mathcal{B}$ and the remaining vertex is from $V \setminus (A \cup \bigcup_{B\in \mathcal{B}} V(B))$.
Similarly, let $y_{A,\mathcal{B}}$ be the number of $k$-cliques in $G$ which $k-1$ vertices form a vertex set of some $B\in \mathcal{B}$ and the remaining vertex is from $A \cup \bigcup_{B\in \mathcal{B}} V(B)$. Finally, let $z_C$ (for $C\subseteq V$) be the number of $k$-cliques containing at least one vertex from $C$.

\begin{prop}\label{prop:auxG_k}
Let $k \geq 3$ be an integer and let $c = \frac{1}{3^{3k}}$. Then there exists a graph $G=(V,E)$ of order $n$ (for sufficiently large $n$) satisfying the following: 
  \begin{enumerate}[(i)]
\item   For every $A \subseteq V$, $|A| \le cn$, and every $\mathcal{B} \subseteq \mathcal{T}_{k-1}(G)$, $|\mathcal{B}|= cn$, vertex disjoint $(k-1)$-cliques such that $A \cap \bigcup_{B\in \mathcal{B}} V(B) = \emptyset$ we have
\[
y_{A,\mathcal{B}} \leq \frac{1}{k+1}x_{A,\mathcal{B}}.
\]
 \label{property:ik}

\item For every $C\subseteq V$, $|C| \le (k-1)cn$,
\[
z_C \le \frac{t_k}{4k}.
\]
\label{property:iiik} 

\item The total number of $k$-cliques satisfies
\[
t_k \le \nu n^{k-1-\alpha}(\log{n})^{1+\alpha},
\]
where $\nu = (3/2)^k \frac{d^{\binom{k}{2}}}{(k-1)(k-2)}.$
\label{property:iik}

  \end{enumerate}
\end{prop}

\begin{proof}
It suffices to show that the random graph $G \in \mathbb{G}(n,p)$ with $p = d(\log n /n)^\beta$ and $d=3000$ satisfies \aas\footnote{An event $E_n$ occurs
\emph{asymptotically almost surely}, or \aas\ for brevity, if
$\lim_{n\rightarrow\infty}\PR{E_n}=1$.}
\eqref{property:ik} - \eqref{property:iik}.

Below we will use the following bounds on the tails of the binomial
distribution $\bin{n,p}$ (for details, see, \emph{e.g.}, \cite{JanLucRuc00}):
\begin{align}
&\Pr(\bin{n,p} \le (1-\gamma)\EX{X})\leq \exp\left( -\frac{\gamma^2}{2}\EX{X}
\right), \label{chernoff_lower_k}\\
&\Pr(\bin{n,p} \ge (1+\gamma)\EX{X})\leq \exp\left( -\frac{\gamma^2}{3}\EX{X}
\right). \label{chernoff_upper_k}
\end{align}

First we show that $G$ \aas\ satisfies \eqref{property:ik}. Fix an $A \subseteq
V$ and $\mathcal{B}\subseteq \mathcal{T}_{k-1}$ with $|\mathcal{B}|=cn$.
Observe that without loss of generality we may assume that $|A|=cn$. Note
that $x_{A,\mathcal{B}}\sim\bin{cn (n -cn -(k-1)cn), p^{k-1}}$. Thus, 
  \[
  \EX{x_{A,\mathcal{B}}}  = c(1-kc)n^2p^{k-1} = d^{k-1}c(1-kc)n^{2-(k-1)\beta}(\log{n})^{(k-1)\beta}
  \]
and \eqref{chernoff_lower_k} (applied with $\gamma=1/2$)  implies
\begin{align} \label{eq:xab_k}
\PR{x_{A,\mathcal{B}} \leq \frac{\EX{x_{A,\mathcal{B}}}}{2}} 
&\leq \exp\left(-\frac{1}{8}\EX{x_{A,\mathcal{B}}}\right)\notag \\
&= \exp\left(-\frac{d^{k-1}}{8}c(1-kc)n^{2-(k-1)\beta}(\log{n})^{(k-1)\beta} \right).
\end{align}
Now we bound from above the number of all possible choices for $A$ and $\mathcal{B}$.
Clearly we have at most $n^{cn}$ choices for $A$. Observe that the number of
choices for $\mathcal{B}$ can be bounded from above by the number of ways of choosing an ordered subset of
vertices of size $(k-1)cn$. Indeed, suppose that   $v_1, \dots, v_{(k-1)cn}$ is
such a choice. Then $\mathcal{B}$ can be defined as $\left\{\{v_1,\dots,v_{k-1}\}, \{v_{k},\dots,v_{2k-2}\},
\dots, \{v_{(k-1)cn - k+1},\dots,v_{(k-1)cn}\}\right\}$. Thus we conclude that there are
at most $n^{kcn}$ ways to choose $A$ and $\mathcal{B}$. Hence, by \eqref{eq:xab_k}
\begin{align} \label{eq:xAB_k}
 \PR{\bigcup_{A,\mathcal{B}} \biggl\{ x_{A,\mathcal{B}} \leq \frac{\EX{x_{A,\mathcal{B}}}}{2} \biggr\} }\notag 
 &\le n^{kcn}\PR{x_{A,\mathcal{B}} \leq \frac{\EX{x_{A,\mathcal{B}}}}{2}} \\ 
 &\le \exp\left( kcn\log{n} - \frac{d^{k-1}}{8}c(1-kc)n^{2-(k-1)\beta}(\log{n})^{(k-1)\beta} \right) \notag \\
 &= o(1).
\end{align}
  
Similarly, since $y_{A,\mathcal{B}} \sim \bin{cn \cdot kcn, p^{k-1}}$,
\[
\EX{y_{A,\mathcal{B}}} = kc^2n^2 p^{k-1} = d^{k-1}kc^2 n^{2-(k-1)\beta} (\log n)^{(k-1)\beta}.
\]
and since $c = \frac{1}{3^{3k}} \le \frac{1}{k(3k+4)}$,
\begin{align*}
\frac{\EX{x_{A,\mathcal{B}}}}{2(k+1)} &= \frac{c(1-kc)}{2(k+1)}d^{k-1} n^{2-(k-1)\beta} (\log n)^{(k-1)\beta} \\
&\ge \frac{3}{2} d^{k-1} kc^2 n^{2-(k-1)\beta} (\log n)^{(k-1)\beta}\\ 
&= \frac{3}{2}\EX{y_{A,\mathcal{B}}}.
\end{align*}
Inequality \eqref{chernoff_upper_k} (applied with $\gamma=1/2$) yields
\[
\PR{y_{A,\mathcal{B}} \geq \frac{\EX{x_{A,\mathcal{B}}}}{2(k+1)}} \leq
\PR{y_{A,\mathcal{B}} \geq \frac{3}{2}\EX{y_{A,\mathcal{B}}}} \leq
\exp\left(-\frac{1}{12}\EX{y_{A,\mathcal{B}}}\right).
\]
Therefore, we deduce that
\begin{align} \label{eq:yAB_k}
 \PR{\bigcup_{A,\mathcal{B}} \biggl\{ y_{A,\mathcal{B}} \geq \frac{\EX{x_{A,\mathcal{B}}}}{2(k+1)} \biggr\} }
 \le  n^{kcn} \exp\left(-\frac{1}{12}\EX{y_{A,\mathcal{B}}}\right) = o(1).
\end{align}

Consequently, by \eqref{eq:xAB_k} and \eqref{eq:yAB_k} we get that \aas
\[
y_{A,\mathcal{B}} \le \frac{\EX{x_{A,\mathcal{B}}}}{2(k+1)} \le \frac{x_{A,\mathcal{B}}}{k+1}
\]
for any choice of $A$ and $\mathcal{B}$. This finishes the proof of~\eqref{property:ik}.

For each vertex $v \in V$, let $\deg_k(v)$ denote the number of $k$-cliques of
$G$ which contain~$v$. In order to show that \aas\ $G$ also satisfies
 \eqref{property:iiik}, we will first estimate $\deg_k(v)$
for each $v \in V$.

The standard application of \eqref{chernoff_upper_k}
(applied with $\bin{n-1,p}$ and $\gamma=1/2$) with the union bound imply that
\aas\ the degree of every vertex $v\in V(G)$ satisfies
\[
\deg(v) \le  \frac{3}{2}d n^{1-\beta} (\log n)^\beta.
\]

The number of $k$-cliques which contain $v$ is equal to the number of
$(k-1)$-cliques in the neighborhood of $v$.  Therefore, in order to show~\eqref{property:iiik} it suffices to bound
the number of $(k-1)$-cliques in any set of size at most 
$\frac{3}{2}d n^{1-\beta} (\log n)^\beta$.

Let $S\subseteq V$ with $s = |S| = \frac{3}{2}dn^{1-\beta} (\log n)^\beta$.  
First we will decompose all $(k-1)$-tuples of $S$ into linear $(k-1)$-uniform hypergraphs
$\mathcal{S}_1, \mathcal{S}_2, \dots, \mathcal{S}_m$ with 
\[
m = (1+o(1))\binom{s}{k-1} \binom{k-1}{2}/\binom{s}{2}
\]
and
\[
|\mathcal{S}_i| = (1+o(1)) \frac{\binom{s}{2}}{\binom{k-1}{2}}
\]
for every $1\le i\le m$.
That means that each $(k-1)$-tuple of $S$ belongs to exactly one $\mathcal{S}_i$ and each pair of elements of $S$ appears in at most one $(k-1)$-tuple in $ \mathcal{S}_i$. The existence of such a decomposition follows from a more general result of Pippenger and Spencer~\cite{PS} (see also \cite{FR}).

Let $s_i$ be the random variable that counts the
number of $(k-1)$-tuples of $\mathcal{S}_i$ which appear as $(k-1)$-cliques of $G$.  Observe that
$s_i \sim
\bin{|\mathcal{S}_i|,p^{\binom{k-1}{2}}}$. Therefore for each $i$,
\begin{align*}
\EX{s_i} &= (1+o(1)) \frac{\binom{s}{2}}{\binom{k-1}{2}} p^{\binom{k-1}{2}} \\
&= (1+o(1))\frac{s^2}{(k-1)(k-2)} p^{\binom{k-1}{2}} \\
&= (1+o(1))\frac{9}{4(k-1)(k-2)}d^{2+\binom{k-1}{2}} n^{1-\beta}(\log{n})^{1+\beta}
\end{align*}
and by \eqref{chernoff_upper_k} (with $\gamma = 1/2$)
\[
\PR{s_i \geq \frac{3}{2}\EX{s_i}} 
\le \exp\left( -\frac{1}{12} \EX{s_i}  \right)
\le \exp\left( -\frac{3}{16k^2}d^{2+\binom{k-1}{2}} n^{1-\beta}(\log{n})^{1+\beta} \right).
\]
Consequently, the union bound over all subsets $S\subseteq V$ of size $s$ and over all $i$ for each $1\le i\le m$ implies 
\begin{align*}
\PR{\bigcup_{S,\,i} \biggl\{ s_i \geq \frac{3}{2}\EX{s_i} \biggr\}} 
&\le \binom{n}{s} \cdot m \cdot \exp\left( -\frac{3}{16k^2}d^{2+\binom{k-1}{2}} n^{1-\beta}(\log{n})^{1+\beta} \right)\\
&\le n^s \cdot s^{k-3} \cdot \exp\left( -\frac{3}{16k^2}d^{2+\binom{k-1}{2}} n^{1-\beta}(\log{n})^{1+\beta} \right)\\
&= s^{k-3} \cdot \exp\left(s\log n -\frac{3}{16k^2}d^{2+\binom{k-1}{2}} n^{1-\beta}(\log{n})^{1+\beta} \right)\\
&= s^{k-3} \cdot \exp\left(n^{1-\beta}(\log{n})^{1+\beta} \left(\frac{3}{2}d -\frac{3}{16k^2}d^{2+\binom{k-1}{2}} \right)  \right)\\
&=o(1),
\end{align*}
since $s^{k-3}$ grows like a polynomial in $n$. Therefore it follows that \aas
\begin{equation}\label{eq:degk}
    \deg_k(v) = \sum_{i=1}^m s_i 
    \le m \cdot \frac{3}{2}\EX{s_i}
    \le s^{k-3} \cdot \frac{3}{2}\EX{s_i}
    = \nu n^{(k-2)(1 - \beta)}(\log{n})^{1+\alpha},
\end{equation}
where
\begin{align}\label{eq:gamma}
\nu = \left( \frac{3}{2} \right)^k \frac{d^{\binom{k}{2}}}{(k-1)(k-2)}.
\end{align}
In a similar way one can show that 
\[
\deg_k(v) \ge \lambda n^{(k-2)(1 - \beta)}(\log{n})^{1+\alpha},
\]
where
\begin{align}\label{eq:lambda}
\lambda = \left( \frac{1}{2} \right)^{k-1} \frac{d^{\binom{k}{2}}}{(k-1)(k-2)}.
\end{align}
{
	Note that equation~\eqref{eq:degk} gives the bound
	\[
	t_k \leq \nu n^{(k-2)(1 - \beta)+1}(\log{n})^{1+\alpha} = 
\nu n^{k - 1 - \alpha}(\log{n})^{1+\alpha},
	\]
	which proves part~\eqref{property:iik}.
}

Now we finish the proof of \eqref{property:iiik}. Since each $k$-clique is counted exactly $k$ times, the number of $k$-cliques is \aas\ at least
\begin{equation}\label{eq:tk}
t_k \ge \frac{n}{k} \cdot \lambda n^{(k-2)(1 - \beta)}(\log{n})^{1+\alpha} = \frac{\lambda}{k} n^{k-1 - \alpha}(\log{n})^{1 + \alpha}.
\end{equation}
It follows now from \eqref{eq:degk} and \eqref{eq:tk} that given a set $C \subseteq V$, $|C| \le (k-1)cn$, the number of $k$-cliques of $G$ which intersect $C$ is \aas\ at most 
\[
z_C \le (k-1)cn \cdot \nu n^{(k-2)(1 - \beta)}(\log{n})^{1+\alpha} 
=  \frac{c(k-1)k\nu}{\lambda} \cdot  \frac{\lambda}{k} n^{k-1 - \alpha}(\log{n})^{1 + \alpha} 
\le \frac{c(k-1)k\nu}{\lambda} t_k.
\]
Finally observe that \eqref{eq:gamma}, \eqref{eq:lambda} together with the choice of $c$ yield that
\[
\frac{c(k-1)k\nu}{\lambda} \le \frac{1}{4k}
\]
implying condition~\eqref{property:iiik}, as required.
\end{proof}

Now we are ready to prove main result of this section.

\bigskip

\noindent
{\bf Proof of Theorem~\ref{thm:tightpath}}. We show that there exists  a $k$-graph $\mathcal{H}$ with\linebreak $|\mathcal{H}| =O(n^{k-1-\alpha}(\log{n})^{1+\alpha})$ such that any two-coloring of the edges of
$\mathcal{H}$ yields a monochromatic copy of $\mathcal{P}_{n,k-1}^{(k)}$.  

Let $G$ be a graph from Proposition~\ref{prop:auxG_k}. Set $V(\mathcal{H}) = V(G)$ 
and let $E(\mathcal{H})$ be the set of $k$-cliques in $G$.  
We prove that such $\mathcal{H}$ is a Ramsey $k$-graph for $\mathcal{P}_{m,k-1}^{(k)}$ with $m=cn$, where $c = \frac{1}{3^{3k}}$.

Take an arbitrary red-blue coloring of the edges of $\mathcal{H}_0 = \mathcal{H}$ and assume that there is no monochromatic $\mathcal{P}_{m,k-1}^{(k)}$.  We will consider the following greedy \emph{procedure} which at each step finds a blue tight path of length $i$ labeled as $v_1,v_2,\dots, v_i$. 

\begin{enumerate}[(1)]

\item Let $\mathcal{B} = \emptyset$ be the \emph{trash} set of $(k-1)$-tuples and $U = V(\mathcal{H})$ be the set of \emph{unused} vertices and set $i:=0$. At any point in the process, if $|\mathcal{B}|=m$, then stop.

\item\label{step:start_k} (In this step $i=0$.) If possible, then choose a blue edge from $U$ and label its vertices by $v_1,\dots,v_k$ and then set $i:=k$. Otherwise, if not possible, stop.

\item\label{step:adding_k} (In this step $i\ge k$.) Let $v_{i-k+1}, \dots, v_{i-1}, v_{i}$ be the labels of the last $k-1$ vertices of the constructed
blue path. If possible, select a vertex $u\in U$ for which
$v_{i-k+1}, \dots, v_{i-1}, v_{i},u$ form a blue edge.  Label $u$ as $v_{i+1}$,
set $U := U\setminus\{u\}$ and $i:=i+1$. Repeat this step until no such $u$ can be found.

\item (In this step also $i\ge k$.) Let $v_{i-k+1}, \dots, v_{i-1}, v_{i}$ be the labels of the last $k-1$ vertices of the constructed
blue path which cannot be extended in a sense described in step \eqref{step:adding_k}. Remove these $k-1$ vertices from the path and set $\mathcal{B} := \mathcal{B} \cup
  \{\{v_{i-k+1}, \dots, v_{i-1}, v_{i}\}\}$ and $i:= i-k+1$. 
  After this removal there
  are two possibilities:
  \begin{enumerate}[(i)]
  \item if $i<k$, then put back $v_1,\dots,v_i$ to $U$ (i.e. $U:=U\cup\{v_1,\dots,v_i\}$), set $i:=0$, and
    return to step~\eqref{step:start_k};
  \item
    otherwise, return to step~\eqref{step:adding_k}.\\
  \end{enumerate}
\end{enumerate}
\noindent
This procedure will terminate under two circumstances: either $|\mathcal{B}|=m$ or no blue edge can be found in step~\eqref{step:start_k}.

First let us consider the case when $|\mathcal{B}|=m$, that means, there are  $m$ vertex disjoint $(k-1)$-tuples
in $\mathcal{B}$. Denote by $A$ the vertex set of the blue path which was obtained when $|\mathcal{B}|=m$.
Clearly, $|A|<m$, otherwise there would be a blue~$\mathcal{P}_{m,k-1}^{(k)}$. We are going to apply Proposition~\ref{prop:auxG_k}
with sets $A$ and $\mathcal{B}$.
Notice that every edge of $\mathcal{H}$ which contains a $(k-1)$-tuple
from $\mathcal{B}$ and the remaining vertex from $V(\mathcal{H}) \setminus (A \cup \bigcup_{B\in \mathcal{B}} B)$ must
be colored red. (This is because for a $(k-1)$-tuple to end up in
$\mathcal{B}$, there must have been no vertex $u$ in step~\eqref{step:adding_k} that could extend the blue path.) It also follows from step \eqref{step:adding_k} that each $(k-1)$-tuple in $\mathcal{B}$ is contained in at least one blue edge.
Thus, Proposition~\ref{prop:auxG_k} \eqref{property:ik} implies that 
$y_{A,\mathcal{B}} \leq \frac{1}{k+1}x_{A,\mathcal{B}}$. That means that
the number of red edges which contain a $(k-1)$-tuple from $\mathcal{B}$ and the remaining vertex from $U$ is at least 
$k+1$ times the number of blue edges with a $(k-1)$-tuple from $\mathcal{B}$.

Now remove all the blue edges from $\mathcal{H}$ which contain a $(k-1)$-tuple from $\mathcal{B}$ and denote such $k$-graph by $\mathcal{H}_1$.
Perform the above procedure on $\mathcal{H}_1$.
This will generate a new trash set $\mathcal{B}_1$.  Observe that $\mathcal{B}_1 \cap \mathcal{B} = \emptyset$, since 
every edge of $\mathcal{H}_1$ which contains a $(k-1)$-tuple from $\mathcal{B}$ must be red.
Again, if $|\mathcal{B}_1| = m$, then we use the same argument as above to find
that the number of red edges in $\mathcal{H}_1$ which contain a $(k-1)$-tuple from $\mathcal{B}_1$ and the remaining vertex from $U$ is at least $k+1$ times the number of blue edges in $\mathcal{H}_1$ with a $(k-1)$-tuple from $\mathcal{B}_1$.
Indeed, we can again apply the inequality from Proposition \eqref{property:ik}. This is because $y_{A,\mathcal{B}_1}$ is smaller than the number of all blue edges in $\mathcal{H}$ containing a $(k-1)$-tuple from $\mathcal{B}_1$, while (since we do not remove red edges) $x_{A,\mathcal{B}_1}$ remains same in both $\mathcal{H}_1$ and $\mathcal{H}$.
Now remove the blue edges from $\mathcal{H}_1$ which contain a $(k-1)$-tuple from $\mathcal{B}_1$ obtaining a $k$-graph $\mathcal{H}_2$. Keep repeating the procedure until it is no longer possible.

At some point, we will run out of blue edges in $\mathcal{H}_j$ for some
$j \geq 1$, and the procedure will terminate prematurely in step~\eqref{step:start_k}.  In this case $|\mathcal{B}_j| <
m$, $|A| = 0$ and $U$ has no blue edges.  However, there still may be
some blue edges which contain a vertex from $\bigcup_{B\in \mathcal{B}_j}V(B)$.
Proposition~\ref{prop:auxG_k}
\eqref{property:iiik} (applied for $C=\bigcup_{B\in \mathcal{B}_j}V(B)$) implies that the number of such
edges is at most 
\[
z_C \le \frac{t_k}{4k}.
\]

Let $x_{A,\mathcal{B}}^i$ and $y_{A,\mathcal{B}}^i$ be the numbers corresponding  to $x_{A,\mathcal{B}}$ and $y_{A,\mathcal{B}}$ obtained at the end of the procedure applied to $\mathcal{H}_i$. Thus, 
\[
y_{A,\mathcal{B}}^i \leq \frac{1}{k+1}x_{A,\mathcal{B}}^i
\]
for each $0\le i\le j-1$. 

Let $t_R$ and $t_B$
denote the number of red and blue edges in $\mathcal{H}$.
Observe that 
\begin{equation}\label{eq:tB}
t_B \le \sum_{0\le i \le j-1} y_{A,\mathcal{B}}^i + z_C
\le \frac{1}{k+1} \sum_{0\le i \le j-1} x_{A,\mathcal{B}}^i  + \frac{t_k}{4k}.
\end{equation}
Furthermore, since all sets $\mathcal{B}_i$ are mutually disjoint,
each red edge in $\mathcal{H}$ containing a $(k-1)$-tuple from some $\mathcal{B}_i$
can be only counted at most $k$ times. Thus,
\begin{equation}\label{eq:tR}
\sum_{0\le i \le j-1} x_{A,\mathcal{B}}^i \le k\cdot t_R.
\end{equation}
Consequently, by \eqref{eq:tB} and \eqref{eq:tR}, we get
\[
t_{k} = t_R + t_B \le t_R + \frac{k}{k+1} t_R + \frac{t_k}{4k}
\]
and so 
\[
t_R \ge \frac{4k-1}{4k}\cdot \frac{k+1}{2k+1} t_k > \frac{1}{2}t_k.
\]
The conclusion is that there are more red edges than there are blue
edges in $\mathcal{H}$.  If we reverse the procedure and look for a red path
instead of a blue one, we will conclude that
there are more blue edges than red edges.  Since these two
statements contradict each other, the only way to avoid both statements
is if a monochromatic path exists.
\qed

\section{Hypergraphs with bounded degree}\label{sec:bounded}

In this section we prove Theorem~\ref{thm:bounddegree}, which states  that  hypergraphs with bounded degree can have nonlinear size-Ramsey numbers. 

\bigskip

\noindent
{\bf Proof of Theorem~\ref{thm:bounddegree}.}
We modify an idea from  R\"odl and Szemer\'edi  \cite{RodSze2000}. For simplicity we only present a proof for $k=3$, which can easily be generalized to $k\ge 3$. The hypergraph $\mathcal{G}$ will be constructed as the vertex disjoint union of graphs $\mathcal{G}_i$ each of which is a tree with a path added on its leaves. Next we will describe the details of such construction. 

Set $c = \frac{1}{5}$. We make no effort to optimize $c$ and always assume that $n$ is sufficiently large. 

Let
\[
t = \left\lfloor  \log_2{ \left( \frac{2\log_2 n}{\log_2\log_2 n} \right) } \right\rfloor.
\]
Consider a binary 3-tree $\mathcal{B}=(V,E)$ on $1+2+4+\dots+2^t$ vertices rooted at vertex~$z$ (see Figure~\ref{fig:B}). Denote by $L(\mathcal{B})$ the set of all its leafs. Call the edge containing $z$ the \emph{root edge}. Observe that 
\begin{equation}\label{eq:V(B)}
|V(\mathcal{B})| = 1+2+4+\dots+2^t = 2^{t+1} - 1 < \log_2 n
\end{equation}
(recall that $n$ is large enough) and 
\[
|L(\mathcal{B})| = 2^t.
\]

Let $\varphi$ by an automorphism of $\mathcal{B}$. Since the root edge $e$ is the unique edge with exactly one vertex of degree 1, $\varphi(z)=z$. The other two vertices of $e$ are permuted by $\varphi$. Consequently, $\varphi$ permutes two vertices of every other edge. Hence, it is easy to observe that the order of the automorphism group of $\mathcal{B}$ satisfies
\[
|Aut(\mathcal{B})| = 2^{ 1+2+4+\dots+2^{t-1}} = 2^{2^t-1} < 2^{2^t}.
\]

\begin{figure}
\begin{center}
\includegraphics{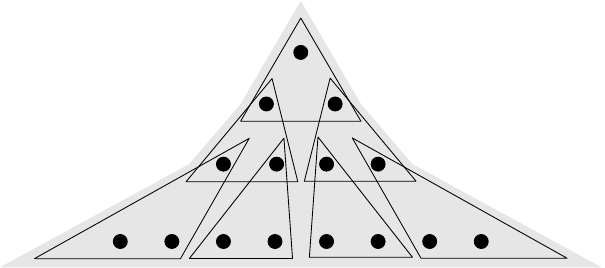}%
\end{center}
\caption{Binary 3-tree $\mathcal{B}$ on $1+2+4+8$ vertices and rooted at vertex $z$.}
\label{fig:B}
\end{figure}

Now consider a tight path $\mathcal{P}$ of length $|L(\mathcal{B})|$ placed on the leaves $L(\mathcal{B})$ in an arbitrary order. Considering labeled vertices of $L(\mathcal{B})$ there are clearly $|L(\mathcal{B})|!$ such paths. Label them by $\mathcal{P}_i$ for $i=1,2,\dots,|L(\mathcal{T})|!$. Let $\mathcal{B}_i$ be vertex disjoint copies of $\mathcal{B}$ and $\mathcal{G}_i = \mathcal{B}_i \cup \mathcal{P}_i$, where $V(\mathcal{P}_i) = L(\mathcal{B}_i)$. 

Let $\varphi$ be an isomorphism between $\mathcal{G}_i$ and $\mathcal{G}_j$. Since the only vertices of degree 4 are on paths $\mathcal{P}_i$ and $\mathcal{P}_j$, $\varphi(\mathcal{P}_i) = \mathcal{P}_j$. Thus, 
\[
\varphi(E(\mathcal{B}_i)) =  \varphi(E(\mathcal{G}_i) \setminus E(\mathcal{P}_i)) = E(\mathcal{G}_j) \setminus E(\mathcal{P}_j) = E(\mathcal{B}_j)
\] 
and so $\mathcal{B}_i$ and $\mathcal{B}_j$ are isomorphic. Thus, the number of pairwise non-isomorphic $\mathcal{G}_i$'s is at least

\[
\frac{|L(\mathcal{B})|!} {|Aut(\mathcal{B})|} 
\ge \frac{(2^t)!}{2^{ 2^t}}
\ge \frac{\left( \frac{2^t}{e} \right)^{2^t}}{2^{2^t}}
\ge \frac{\left(2^{t-2} \right)^{2^t}}{2^{2^t}}
= 2^{(t-3)2^t}
> n.
\]

Set 
\[
q = \left\lfloor \frac{n}{|V(\mathcal{B})|} \right\rfloor
\] 
and let $\mathcal{G} = \mathcal{G}_1 \cup \dots \cup \mathcal{G}_q$, where all $\mathcal{G}_1,\dots, \mathcal{G}_q$ are pairwise non-isomorphic. We show that $\mathcal{G}$ is a desired hypergraph. 

Clearly, $|V(\mathcal{G})|\le n$. Furthermore, by \eqref{eq:V(B)}, we get
\[
|V(\mathcal{G})| = q |V(\mathcal{B})| \ge \left( \frac{n}{|V(\mathcal{B})|} - 1 \right) |V(\mathcal{B})| > n-\log_2 n.
\]
Moreover, $\Delta(\mathcal{H}) = 4$ and the independence number of $\mathcal{G}$ satisfies 
\begin{equation}\label{eq:alpha}
\alpha(\mathcal{G}) \le \frac{8}{9}n.
\end{equation}
Indeed, let $I \subseteq V = V(\mathcal{G})$ be an independent set of size $\alpha = \alpha(\mathcal{G})$. We estimate the number of edges $e(I, V\setminus I)$ between sets $I$ and $V\setminus I$. First observe that 
\[
e(I, V\setminus I) \le \Delta(\mathcal{G}) \cdot |V\setminus I| \le 4 (n-\alpha).
\]
Next, since each triple between $I$ and  $V\setminus I$  intersects one of the partition classes on 2 vertices and $\delta(\mathcal{G})=1$,
\[
e(I, V\setminus I) \ge \frac{\delta(\mathcal{G}) \cdot |I|}{2} = \frac{\alpha}{2}.
\]
This implies that 
\[
\frac{\alpha}{2} \le 4 (n-\alpha)
\]
and so \eqref{eq:alpha}.

Now we are ready to finish the proof and show that for any 3-graph with
\[
|E(\mathcal{H})| \le \frac{1}{30} n (\log_2 n)^{\frac{1}{5}}
\]
we have $\mathcal{H}\nrightarrow \mathcal{G}$.

Set $d = (\log_2 n)^{\frac{1}{5}}$ and define $V_{high} \subseteq V(\mathcal{H})$ as
\[
V_{high} = \{ v\in V(\mathcal{H}) : \deg(v) \ge  d\}
\]
and
\[
V_{low} = V(\mathcal{H}) \setminus V_{high}.
\]
Clearly, $|V_{high}| \le \frac{n}{10}$; for otherwise, $|E(\mathcal{H})| > \frac{n}{10} \cdot d \cdot \frac{1}{3} \ge |E(\mathcal{H})|$, a contradiction.

Recall that $\mathcal{G}$ consists of $q$ pairwise non-isomorphic copies of $\mathcal{G}_i$. We estimate the number of copies of $\mathcal{G}_i$'s contained in a sub-hypergraph induced by $V_{low}$. First fix an edge $e$ in $V_{low}[\mathcal{H}]$ and count the number of copies of $\mathcal{G}_i$'s for which $e$ is a root edge. Since $\deg(v) \le d$ for each $v\in V_{low}$, we get that this number is at most
\[
3 \cdot d^{2+4+\dots+2^{t-1}} \cdot d^{2^t} 
\le d^{2\cdot 2^t}
\le (\log_2 n)^{\frac{1}{5} \cdot 2\cdot \frac{2 \log_2 n}{\log_2 \log_2 n}}
= n^{\frac{4}{5}},
\]
where the factor 3 counts the number of choices for the root vertex,  the next factors count the number of possible $\mathcal{B}_i$'s with $e$ as a root, and the last factor counts the number of paths on the set of leafs. Thus, there is an $i_0$ such that $\mathcal{G}_{i_0}$ appears in $V_{low}[\mathcal{H}]$ at most
\[
\frac{n^{\frac{4}{5}} \cdot |E(\mathcal{H})|}{q}
< \frac{n^{\frac{4}{5}} \cdot n (\log_2 n)^{\frac{1}{5}}}{\frac{n}{\log_2 n}}
= n^{\frac{4}{5}} (\log_2 n)^{\frac{6}{5}}
\]
times. 

Denote by $\mathcal{F}$ the sub-hypergraph consisting of root edges from all copies of $\mathcal{G}_{i_0}$ in $V_{low}[\mathcal{H}]$. Thus, 
\[
|V(\mathcal{F})| \le 3n^{\frac{4}{5}} (\log_2 n)^{\frac{6}{5}}.
\]
Color edges in $\mathcal{F}$ together with edges incident to $V_{high}$ blue; otherwise red. Clearly, there is no red copy of $\mathcal{G}$, since there is no red copy of $\mathcal{G}_{i_0}$. Moreover, there is no blue copy of $\mathcal{G}$, since every blue sub-hypergraph of order $|V(\mathcal{G})|$ has an independent set of size at least 
\[
|V(\mathcal{G})| - |V_{high}| - |V(\mathcal{F})|
>  (n-\log_2 n) - \frac{n}{10} - 3n^{\frac{4}{5}} (\log_2 n)^{\frac{6}{5}} = \frac{9}{10}n- \log_2 n - 3n^{\frac{4}{5}} (\log_2 n)^{\frac{6}{5}},
\]
which is strictly bigger than $\alpha(\mathcal{G})$ (\emph{cf.} \eqref{eq:alpha}).
\qed


\providecommand{\bysame}{\leavevmode\hbox to3em{\hrulefill}\thinspace}

\end{document}